\documentclass[microtype,12pt]{article} 
\usepackage{custarticle}

\date{27 AUG 2022}

\title{Seiberg--Witten monopoles and flat $\PSL(2,\R)$-connections}
\author{Andriy Haydys}

\begin{document}

\maketitle

\begin{abstract}
	I show that flat $\PSL(2;\R)$-connections on three-manifolds satisfying certain `stability condition' can be interpreted as solutions of the Seiberg--Witten equations with two spinors. 
	This is used to construct explicit examples of the Seiberg--Witten moduli spaces.
	Also, I show that  in this setting blow up sets satisfy certain non-trivial topological restrictions. 
\end{abstract}

\section{Introduction}

Studies of the boundary points of the moduli space of solutions to the Seiberg--Witten equations with multiple spinors   are concerned with some interesting phenomena, which are new and of importance beyond the Seiberg--Witten theory.
Similar phenomena occur for example in the studies of the Vafa--Witten, Kapustin--Witten, complex anti-self-duality, Hermitian Yang--Mills, $\rG_2$- and $\Spin(7)$-instanton equations and many other interesting geometric PDEs.

To set the stage, let $M$ be a closed oriented Riemannian three-manifold. 
Pick a spin structure and denote by $\slS$ the corresponding spinor bundle. 
Recall that $\slS$ is a Hermitian rank $2$ bundle such that $\End_0 (\slS)$ is isomorphic to $T^*_\C M = T^*M\otimes\C$, where the subscript $0$ indicates the subbundle of trace-free endomorphisms. 

Let $E$ be any fixed Hermitian vector bundle of rank $2$ such that $\Lambda^2 E$ is trivial, so that the structure group of $E$ is $\SU(2)$.
Fix also an $\SU(2)$-connection $b$ on $E$.

For a Hermitian line bundle $\cL$ the bundle $\Hom\bigl ( E; \slS\otimes \cL \bigr )$ will be referred to as the twisted spinor bundle. 
If $\Psi$ is a section of the twisted spinor bundle, then 
\begin{equation}
	\label{Eq_Mu}
	\mu(\Psi) = \Psi\Psi^* - \frac 12 |\Psi|^2
\end{equation}
is a trace-free Hermitian endomorphism of $\slS$ (the twist by $\cL$ is immaterial here).
Using the isomorphisms $\End_0 (\slS)\cong T^*_\C M\cong \Lambda^2 T^*_\C M$, $\mu(\Psi)$  can be identified with a purely imaginary 2-form on $M$. 
With this at hand, the Seiberg--Witten equations with two spinors read
\begin{equation}
	\label{Eq_SW2gen}
	\slD_{a\otimes b}\Psi =0\qandq\quad F_a=\mu(\Psi),
\end{equation} 
where $a$ is a Hermitian connection on $\cL$, see~\cite{HaydysWalpuski15_CompThm_GAFA} for more details.
Let me just point out that while \eqref{Eq_SW2gen} looks just like the classical Seiberg--Witten equations, an essential difference lies in the structure of the quadratic map $\mu$, which  in a local trivialization of $E$ can be written as  follows
\[
(\psi_1,\psi_2)\mapsto \psi_1\psi_1^* - \frac 12 |\psi_1|^2 \ + \ \psi_2\psi_2^* - \frac 12 |\psi_2|^2.
\]
In particular, $\mu$ is no longer proper, which, in some sense, is a source of  potential non-compactness of the moduli space of solutions. 
The failure of the compactness is discussed in detail in~\cite{HaydysWalpuski15_CompThm_GAFA}. 
In fact, I construct below fairly explicit examples of such  moduli spaces, in particular some non-compact ones, see the discussion at the end of Section~\ref{Sect_Involutions}.

According to~\cite{HaydysWalpuski15_CompThm_GAFA}, if $(a_k, \Psi_k)$ is a non-convergent sequence of solutions of~\eqref{Eq_SW2gen}, then a subsequence of $\big (a_k, \|\Psi_k \|_{L^2}^{-1}\Psi_k \big )$  converges to some $(a,\Psi)$ over $M\setminus Z$, where $Z$ is a closed subset of $M$ of Hausdorff dimension at most one and $(a,\Psi)$ satisfies
\begin{equation}
	\label{Eq_SW2genLim}
	\slD_{a\otimes b}\Psi =0\qandq \mu(\Psi) =0\qquad\text{ over } M\setminus Z.
\end{equation}
Moreover, the pointwise norm of $\Psi$ extends as a continuous function to all of $M$,  $|\Psi|^{-1}(0) = Z$, $\|\Psi\|_{L^2(M\setminus Z)} =1$, and $a$ is flat with the monodromy in $\{ \pm 1\}$.  
Thus, in a certain sense solutions of~\eqref{Eq_SW2genLim} describe the topological boundary of the moduli space of solutions of the Seiberg--Witten equations with two spinors.
Notice that a solution of ~\eqref{Eq_SW2genLim} does not need to be defined along $Z$. 

It follows from the proof of~\cite{HaydysWalpuski15_CompThm_GAFA}*{Thm.\,1.1}  that if the sequence $\big (a_k,\|\Psi_k \|^{-1}\Psi_k\big )$ converges to $(a,\Psi)$ in the sense described above, then the set
\begin{equation*}
\Bigl\{ m\in M \ \ \big |\ \  \exists r_k\to 0\ \text{ s.t. }\ r_k\int_{B_{r_k}(m)} |F_{A_k}|^2  \to\infty \Bigr \}
\end{equation*}
is contained in $Z$, where $B_r(m)$ is the geodesic ball of radius $r$ centered at $m$.
This motivates the following.
\begin{defn}[\cite{Haydys19_InfinitMultipl_AiM}]
	\label{Defn_BlowUpSet}
	A closed nowhere dense set $Z\subset M$ is called a blow-up set for the Seiberg--Witten equation with two spinors, if there is a  solution $(a, \Psi)$ of~\eqref{Eq_SW2genLim} defined over $M\setminus Z$ such that the following holds:
	\begin{enumerate}[(i)]
		\item
		\label{It_HoelderCont}
		$|\Psi|$ extends as a H\"older-continuous function to all of $M$ and $Z=|\Psi|^{-1}(0)$;
		\item 
		\label{It_NablaPsi}
		$\int_{M\setminus Z} |\nabla^A\Psi|^2<\infty$.
	\end{enumerate}  
\end{defn}

Condition~\ref{It_NablaPsi} in the above definition is of technical nature and will not be used directly in the discussion below. 

It is easy to see that if the determinant line bundle $\cL^2$ is non-trivial, then $Z\neq \varnothing$. 
Indeed, assume that for a non-trivial $\cL^2$ there is a solution $(A,\Psi)$ of~\eqref{Eq_SW2genLim} such that $\Psi$ vanishes nowhere, i.e., $Z=\varnothing$. 
The equation $\mu(\Psi)=\Psi\Psi^*-\tfrac 12 |\Psi|^2=0$ implies that $\ker\Psi^*=\{ 0 \}$ pointwise.
Hence, $\Psi$ is surjective everywhere, which in turn yields that $\Psi$ is an isomorphism, because $\rk E = \rk \big (\slS\otimes \cL \big )$. 
Therefore, $\Lambda^2 E\cong \Lambda^2(\slS\otimes\cL)\cong \cL^2$, which yields that $\cL^2$ is trivial thus providing a contradiction.

Thus, in order to construct a compactification of the moduli space of the Seiberg--Witten monopoles with two spinors one needs to understand properties of blow up sets $Z$. 
In particular, one can ask whether there are any other restrictions on $Z$ apart from being closed and of Hausdorff dimension at most one.

A topological restriction for $Z$ has been established in~\cite{Haydys19_InfinitMultipl_AiM}. 
Namely, it has been shown that $Z$ supports a homology class, which is  Poincar\'e dual to $c_1(\cL^2)$. 

In this manuscript a  topological restriction of another type is obtained.
Before explaining this, let me note that in general $Z$ does not need to be  smooth, see however~\cite{Zhang17_Rectifiability_Arx} for the most general regularity statement currently known.
Recent results by Taubes--Wu~\cite{TaubesWu20_SingularityModels_Arx} yield a strong evidence that $Z$ may be singular in general. 
Nevertheless, it is conceivable that after a suitable perturbation (of the background metric, say) the blow up set will become a smoothly embedded $1$-submanifold, i.e., a link in $M$. 
In any case, to formulate the next result, let me \emph{assume} that $Z$ is a link indeed.

Thus, put $E=\slS$ with $b$ being the Levi--Civita connection and assume that  $Z$ is a link in $M$. 
 Denote by  $Z_1$  the union of all components of $Z$ such that the monodromy of $a$ along the meridian is non-trivial.

  \begin{thm}
  	\label{Thm_AlexPolyBlowUpSet}
  	Let $(a,\Psi, Z)$ be a solution of~\eqref{Eq_SW2genLim}  on an integral homology sphere $M$ with $E=\slS$ and $b$ being the Levi--Civita connection.
  	If $Z$ is smooth and $\Delta_{Z_1} (t)$ denotes the  Alexander polynomial of $Z_1$, then $\Delta_{Z_1}(-1)=0$. 
  	In particular, $Z$ consists of at least two connected components if it is smooth.
  \end{thm}
 
 The proof of this theorem can be found on Page~\pageref{Pf_AlexPoly}. 
  Concerning the very last claim about the disconnectedness of $Z$, it is well-known that the  Alexander polynomial of a 1-component link does not vanish at the point $-1$, see for example~\cite{Lickorish97_IntroKnotThry}*{Cor.\,6.11} in the case $M=S^3$. 
  In any case, this will be also clear from the proof of  \autoref{Thm_AlexPolyBlowUpSet}.

 While a very particular choice of the twist is required for the proof of \autoref{Thm_AlexPolyBlowUpSet}, the topological restrictions obtained are  stable under small deformations. 
 To explain, notice first that solutions of~\eqref{Eq_SW2genLim} correspond~\cite{HaydysWalpuski15_CompThm_GAFA}*{App.\,A} to certain $\Z/2$ harmonic 1-forms~\cite{Taubes13_PSL2Ccmpt, Taubes15_CorrigendumToPSL2C}, for which $Z$ is a part of the data. 
 In any case, if $Z$ is an embedded link, then the space of all $\Z/2$ harmonic 1-forms in some neighborhood is cut out by a Fredholm map~\cite{Takahashi15_Z2HarmSpinors_Arx}.
 Hence, it is reasonable to expect that if $(a, \Psi, Z)$ is a non-trivial solution of~\eqref{Eq_SW2genLim} for $(g,b) = \big (g,\nabla^{LC}\big )$, then for any other choice of parameters $(\hat g, \hat b)$, which are sufficiently close to $\big (g,\nabla^{LC}\big)$ and admit a non-trivial solution $(\hat a, \hat\Psi,\hat Z)$,  the set $\hat Z$ is also an embedded link, which is isotopic to $Z$.
 In this case, $\Delta_{\hat Z_1}(-1)$ must vanish too.   
 In this sense, the conclusion of \autoref{Thm_AlexPolyBlowUpSet} is a manifestation of generic properties of blow up sets. 
  
 \medskip
 
 A crucial step in the proof of  Theorem~\ref{Thm_AlexPolyBlowUpSet} is a correspondence between solutions of~\eqref{Eq_SW2gen} and \eqref{Eq_SW2genLim} with flat $\PSL(2,\R)$-connections on $M$ and $M\setminus Z$ respectively satisfying certain stability condition, see Definition~\ref{Def_FlatStablePSL2R} and Lemmas~\ref{Lem_FlatPSL2RRedSW} and~\ref{Lem_SW2PSL2RBijection} below.  
 Combining this with known topological and analytic results, the proof of the above theorem follows quite easily.
 
 \medskip
 
 \textsc{Acknowledgment.} I am grateful to the Simons Foundation for the financial support and Stanford University for hospitality, where part of this project has been carried out. 
 Also, I wish to thank an anonymous referee for helpful questions and comments.

\section{Basic constructions}
\label{Sect_BasicConstr}

To fix notations, for a principal $\rG$-bundle $P\to M$, denote by $\cA(P)$ the space of all connections on $P$ and recall that $\cA(P)$ is an affine space modeled on $\Om^1(ad\, P)$, where $ad\, P:=P\times_{\rG, ad}\frak g$ with $\fg$ being the Lie algebra of $\rG$.
The gauge group 
\begin{equation*}
	\cG(P)=\big \{  \psi\colon P\to P\mid \pi\comp\psi = \pi,\quad \psi(p\cdot g) = \psi(p)\cdot g \big\}
\end{equation*}
 acts naturally on $\cA(P)$  by pull-backs.

Recall also that the group $\PSL(2,\C) = \SL(2,\C)/\pm 1$ acts transitively on the hyperbolic space
\begin{equation*}
	\mathbb H^3:= \big \{ H\in M_2(\C)\mid H^* = H,\ \det H =1 \big \}\qquad\text{by}\qquad g\cdot H= gHg^*.
\end{equation*}
Since the stabilizer of the identity matrix is $\PSU(2)=\SU(2)/\pm 1\cong \SO(3)$, we have $\mathbb H^3=\PSL(2,\C)/\PSU(2)$. 

\medskip

Let $\pi\colon Q^c\to M$ be a principal $\PSL(2,\C)$-bundle, where $M$ is a closed oriented Riemannian three manifold just as in the preceding section. 
Choose a reduction of the structure group to $\PSU(2)$, that is a $\PSU(2)$-subbundle $Q\subset Q^c$. 
Recall that such reduction corresponds to a section $s$ of the bundle
\begin{equation*}
	Q^c \times_{\PSL(2,\C)} \PSL(2,\C)/\PSU(2)=Q^c\times_{\PSL(2,\C)} \mathbb H^3
\end{equation*} 
or, equivalently, an equivariant map $\hat s\colon Q^c\to \mathbb H^3$. 
Explicitly,  $\hat s$ and $Q$ are related by
$Q=\big \{ q\in Q^c\mid \hat s(q) = \mathbbm 1  \big \}$.
Notice that $Q^c$ can be recovered from $Q$, since  
 $Q^c=Q\times_{\mathrm{\PSU}(2)}\PSL(2,\C)$.

For any connection $A$ on $Q^c$, which is an equivariant 1-form on $Q^c$ with values in $\fpsl (2,\C)$, write $A=a + (\pi^*b)i$, where $a$ is a connection on $Q$ while $b$ is a 1-form on $M$ with values in $\ad Q$.   
Then  $0=F_A= F_a + (d_a b)i -\frac 12 [b\wedge b]$, that is $A$ is flat if and only if 
\begin{equation}
 \label{Eq_FlatPSL2C}
d_a b=0\quad\text{and}\quad F_a=\frac 12 [b\wedge b],
\end{equation}
cf.~\cite{Hitchin:87}.
Notice that the space of solutions of these equations is invariant under the action of $\cG(Q^c)$, which is referred to as \emph{the complex gauge group} below.

\begin{remark}
Sometimes I identify  $b\in \Om^{1}(\ad Q)$ with $\pi^*b\in \Om^1(Q;\su(2))$ dropping the pull-back from the notations if this is unlikely to lead to a confusion.   
\end{remark}

\begin{defn}
	I  say that $A=a + (\pi^*b)\, i$ is a \emph{stable flat $\PSL(2,\C)$-connection}, if  $A$ is flat and the following condition holds: 
	\begin{equation}
	\label{Eq_Stability}
	d_a^* b = -*d_a\! * b=0.
	\end{equation}	
\end{defn}
 
 	The above stability condition has been studied at least  starting from~\cite{Hitchin:87, Donaldson:87_TwistedHarmonicMaps, Corlette:88_FlatGbundles} and  can be understood as follows.
 	By writing connections on $Q^c$ as pairs just like above, we have an isomorphisms $\cA(Q^c)\cong \cA(Q)\times \Om^1(\mathrm{ad}\, Q)\cong T^*\cA(Q)$.
 	In particular, $\cA(Q^c)$ has a natural symplectic structure.  
 	The gauge group $\cG(Q)$, which is referred to as a `real gauge group' in the sequel, acts in a Hamiltonian fashion, and the corresponding moment map can be identified with 
 	\[
 	\cA(Q)\times \Om^1(\mathrm{ad}\, Q)\to \mathrm{Lie}\bigl (\cG(Q)\bigr )\cong \Omega^0(\mathrm{ad}\, Q),\qquad 
 	(a,b)\mapsto d_a^* b.
 	\]
 	Hence, \eqref{Eq_Stability} demands that $(a,b)$ lies in the zero locus of this moment map.
 	 This is the familiar `stability condition' from  algebraic/symplectic geometry.
 	 
 	 Notice also that~\eqref{Eq_Stability} is  preserved by \emph{the real gauge group} $\cG(Q)$, but not by the complex one.

\begin{defn}
	\label{Def_FlatStablePSL2R}
	I say that a pair $(A, Q^r)$ is a \emph{flat stable $\PSL(2,\R)$-connection}, if the following holds: 
	\begin{itemize}[itemsep=0pt]
		\item $A=a+(\pi^*b)\,i$ is a solution of~\eqref{Eq_FlatPSL2C} and \eqref{Eq_Stability};
		\item $Q^r\subset Q^c$ is a $\PSL(2,\R)$-subbundle  such that $A$ reduces to $Q^r$, i.e., the restriction of $A$ to $Q^r$ takes values in $\fpsl(2,\R)$.
	\end{itemize}
\end{defn}

\begin{remark}
	\label{Rem_UniquenessPSL2Rsubbundle}
	For any $\PSL(2,\C)$-connection $A$ and any $q_0\in Q^c$, let $\Hol(A, q_0)$ be the holonomy group of $A$ relative to $q_0$.
	Let $Q_{\mathrm{hol}}(A,q_0)$ denote the holonomy bundle of $A$, 
	that is $Q_{\mathrm{hol}}(A,q_0)$ consists of all those $q\in Q^c$ which can be connected with $q_0$ by a horizontal curve.
	Then $Q_{\mathrm{hol}}(A,q_0)$ is a principal $\Hol(A,q_0)$-bundle and it is well known that $A$ reduces to $Q_{\mathrm{hol}}(A,q_0)$, see for example~\cite{Nomizu55_ReductionThm}*{Prop.\,2}.
	Then for any $g\in \PSL(2,\C)$ we have
	\begin{equation*}
	\Hol(A, q_0 g) =g^{-1} \Hol(A, q_0) g\qquad\text{and}\qquad 
	Q_{\mathrm{hol}}(A,q_0g) = Q_{\mathrm{hol}}(A,q_0)g.
	\end{equation*}
	
	If $\Hol (A, q_0)\subset \PSL(2;\R)$, then there exists a unique  $\PSL(2,\R)$-bundle $Q^r(A, q_0)\supset \Hol(A,q_0)$ such that $A$ restricts to $Q^r(A, q_0)$. 
	Any other choice of the basepoint yields
	\begin{equation*}
	Q_{\mathrm{hol}}(A,q_0\cdot g) = Q_{\mathrm{hol}}(A,q_0)\cdot g\qquad\implies\qquad 
	Q^r(A,q_0\cdot g) = Q^r(A,q_0)\cdot g,
	\end{equation*} 
	where $g\in \PSL(2, \C)$.
	
	Hence, if $g_0$ lies in the normalizer of $\PSL(2,\R)$ and $A$ reduces to $Q^r$, then $A$ also reduces to $Q^r\cdot g_0$. 
	Notice that in this case $\Hol(A,q_0)$ and $\Hol(A,q_0\cdot g_0)$ are related by an outer automorphism of $\PSL(2,\R)$. 
	I shall come back to this point in Section~\ref{Sect_Involutions} again. 
\end{remark}

Notice that in the above definition I fix the standard embedding $\PSL(2,\R)\subset \PSL(2,\C)$. 
Since $\PSL(2,\R)\cap \PSU(2)=\U(1)$,  this yields a distinguished copy of $\U(1)$ both in $\PSU(2)$ and $\PSL(2,\R)$. 
In the latter case, this copy of $\U(1)$ is a preferred maximal compact subgroup.

 The reduction of the structure group of $Q^c$ to $\PSU(2)$ induces a reduction of the structure group of $Q^r$ to $\U(1)$. 
 Indeed,  the corresponding $\U(1)$-subbundle is simply
 \begin{equation*}
 	P:=\big \{ p\in Q^r\mid \hat s(p) = \mathbbm 1 \big \} = Q^r\cap Q.
 \end{equation*}

Let $\cL$ be the complex Hermitian line bundle associated with $P$ and the standard $\U(1)$-rep\-re\-sen\-ta\-tion. 
Denote by $\chi$ the fiberwise symplectic form on $\cL$. 
A combination of the wedge-product and $\chi$ yields a fiberwise quadratic map $\chi(\cdot\wedge \cdot)\colon Sym^2\bigl ( T^*M\otimes \cL\bigr )\to \Lambda^2T^*M$. 

Since $P$ can be also viewed as a reduction of the structure group of $Q$ to $\U(1)$,  the bundle $\ad Q$ splits into the trivial bundle of rank $1$ and a bundle of rank $2$. 
The latter is naturally isomorphic to $\cL$.

\begin{lem}
	\label{Lem_FlatPSL2RRedSW}
	For each flat stable $\PSL(2,\R)$-connection there exists a unique triple $(\cL, a, b)$, where $\cL$ is a Hermitian line bundle, $a$ is a Hermitian connection on $\cL$, and $b$ is a 1-form with values in $\cL$, such that  $(a,b)$ satisfies
	\begin{equation}
	\label{Eq_SFSl2Creduced}
	(d_a + d_a^*) b=0\qquad\text{and}\qquad 	F_a=\chi(b\wedge b) i.
	\end{equation} 
	
	Conversely, for any triple $(\cL, a, b)$ as above such that $(a, b)$ satisfies~\eqref{Eq_SFSl2Creduced} there is a unique  principal $\PSU(2)$-bundle $Q:=P\times_{\U(1)}\PSU(2)$ and a unique flat stable $\PSL(2,\R)$-connection with 
	\begin{equation}
		\label{Eq_QrQc}
		Q^r:=P\times_{\U(1)}\SL(2,\R)\;\subset\; P\times_{\U(1)}\PSL(2,\C)=:Q^c.
	\end{equation}
\end{lem}
\begin{proof}
	Let $\fu (1)\subset \fpsl(2;\R)$ be the Lie algebra of the maximal compact subgroup $\U(1)\subset \PSL(2;\R)$. 
	Notice that with respect to the decomposition  $\fsl(2,\C)=\su(2)\oplus\su(2)i$ we have $\fpsl(2;\R) = \fu(1)\oplus \fu(1)^\perp i$, where $\fu(1)^\perp$	is the orthogonal complement of $\fu(1)$ in $\su(2)$. 
	Choose an orthonormal oriented basis $(\eta_1,\eta_2,\eta_3)$ of $\su(2)$ such that $\fu(1)=\R\eta_1$ and $\mathrm{span}(\eta_2,\eta_3)=\fu(1)^\perp$, say 
	\begin{equation}
	\label{Eq_BasisSU2}
	\eta_1=
	\begin{pmatrix}
	0 & -1\\ 1 & \phantom{-}0
	\end{pmatrix},
	\quad
	\eta_2=
	\begin{pmatrix}
	\phantom{-}0 & -i\\ -i & \phantom{-}0
	\end{pmatrix},
	\qand\quad
	\eta_3=
	\begin{pmatrix}
	i & \phantom{-}0\\ 0 & -i
	\end{pmatrix}.
	\end{equation}
	Notice that there is a unique symplectic form $\chi$ on $\fu(1)^\perp$ such that $\chi(\eta_2,\eta_3) =1$, that is  $(\eta_2,\eta_3)$ is a symplectic basis of $\fu(1)^\perp$.

	If $Q^r\subset Q^c$ is a principal $\PSL(2,\R)$-subbundle as above, then  $A=a+(\pi^*b)i$ reduces to $Q^r$ if and only if the real part takes values in the one-dimensional subspace $\fu(1)\subset \su(2)$, while the imaginary part takes values in $\fu (1)^\perp i$.
	In other words, $a$ is a connection on $P = Q^r\cap Q$ and $b$ is a one-form on $M$ with values in $\cL=P\times_{\U(1)}\C$. 
	
	Furthermore,  the splitting $\su(2) = \fu(1) \oplus \fu(1)^\perp$ yields the decomposition $\ad Q^r=\underline{\R}\oplus \cL$.
	Since $[\eta_2,\eta_3] = 2\eta_1$, 
	the restriction of the Lie-brackets to $\cL$ takes values in $\underline{\R}$ and equals $2\chi$.
	Hence, the real part of $F_A$ vanishes if and only if $F_a = \chi(b \wedge b)i$.
	Thus, a $\PSL(2,\R)$-connection $A=a+(\pi^*b)i$ is flat and stable if and only if~\eqref{Eq_SFSl2Creduced} holds.

	 \medskip
	
	Conversely, given a Hermitian line bundle $\cL$ together with a solution $(a,b)$ of~\eqref{Eq_SFSl2Creduced}, essentially the same computation (reading backwards) yields that $A=a+ (\pi^*b)i$ is a framed flat stable $\PSL(2;\R)$-connection with $Q^r$ given by~\eqref{Eq_QrQc}.
\end{proof}

\begin{example}
	Let $\cL$ be the product line bundle $\underline \C:=M\times\C$ and $a=\vartheta$ the product  connection. 
	Assume furthermore that $b$ is a real-valued harmonic 1-form.
	Then for any $w\in \C$ the pair $(\vartheta, wb)$ is clearly a solution of~\eqref{Eq_SFSl2Creduced}. 
	In this case the holonomy of $A$ is contained in a one-parameter subgroup generated by an element of $\fu(1)^\perp\otimes i\subset \fsl(2,\R)$.
	In particular, $\Hol(A)$ is abelian.  
\end{example}

Starting from a different perspective, consider the Seiberg--Witten equations~\eqref{Eq_SW2gen} with $E=\slS$, which is equipped with the Levi--Civita connection. 
From now on I will assume this particular twist throughout even if this is not mentioned explicitly. 
In this case we have a well-defined trace map
 \[
 \tr\colon\Hom(\slS,\slS\otimes \cL)\cong \End(\slS)\otimes\cL\to \cL.
 \]
Denote by $\Hom_0(\slS,\slS\otimes\cL)\cong \End_0(\slS)\otimes\cL$ the subbundle of traceless homomorphisms. 	
The Clifford multiplication (twisted by the identity map on $\cL$) provides an isomorphism
\[
\Cl\colon T^*_\C M\otimes \cL\to \End_0(\slS)\otimes \cL, 
\]
which in turn yields an isomorphism 
\[
\Upsilon\colon\Hom (\slS, \slS\otimes\cL)\longrightarrow T^*_\C M\otimes \cL \oplus \cL,
\]
where $\Upsilon = (\Cl^{-1}, \tr )$ and $\Cl^{-1}$ is extended trivially to the trace-component.

\begin{lem}
	\label{Lem_SW2PSL2RBijection}
	A pair $(a,\Psi)$ such that $\tr\Psi =0$ is a solution of~\eqref{Eq_SW2gen} with $E=\slS$ and $b$ being the Levi-Civita connection  if and only if $(a, b)=\left ( a, \Cl^{-1}(\Psi)  \right)$ solves~\eqref{Eq_SFSl2Creduced}.
	Moreover, the following holds:
	\begin{enumerate}[(i)]
		\item \label{It_LnontrivialImpliesTrace0}
		If $\cL$ is non-trivial, then for any solution of~\eqref{Eq_SW2gen} we have $\tr \Psi =0$;
		\item Any solution $(a,\Psi)$ of~\eqref{Eq_SW2gen} with $\tr \Psi\neq 0$ is gauge-equivalent to a pair $(\vartheta, \om)$, where $\vartheta$ is the product connection on $\cL=\underline{\C}$ and $\om$ is a purely imaginary harmonic 1-form.
		In particular, in this case $(a,\Psi)$ corresponds to a flat $\PSL(2,\R)$-connection with an abelian holonomy.
	\end{enumerate}
\end{lem}
\begin{proof}
	Let $(a,\Psi)$ be a solution of~\eqref{Eq_SW2gen} such that $\tr\Psi=0$. 
	A straightforward computation, whose details can be found in Appendix~\ref{Sect_AlgPtyMu},  yields that $\mu(\Psi) = \Cl\big (\chi(b\wedge b)i\,\big )$ so that $(a,b)$ solves~\eqref{Eq_SFSl2Creduced} indeed.

	Assume now that $s:=\tr\Psi\neq 0$. 
	It follows from~\eqref{Eq_SW2gen} that $s$ is a $\nabla^a$-covariantly constant section.
	Since $a$ is Hermitian, $s$ vanishes nowhere, hence proving~\textit{\ref{It_LnontrivialImpliesTrace0}}.
	Furthermore,  $a$ is the product connection \wrt the trivialization given by $s$.  
	Just like above, the traceless component $\Psi_0$ of $\Psi$ can be identified with some complex-valued 1-form $b_1$ and~\eqref{Eq_SW2gen} translates into  
	\begin{equation*}
		(d+ d^*)\, b_1 =0\qquad\text{and}\qquad	
		F_a = \chi (b_1\wedge b_1)+ 2*\Re b_1,
	\end{equation*}
	where we also have $F_a=F_\vartheta =0$. 
	
	Furthermore, writing $b_1=b_{10} + b_{11}i$ we obtain $b_{10}\wedge b_{11} + 2*b_{10}=0$, which yields in turn
	\[
	2\,|b_{10}|^2 = 2\,b_{10}\wedge *b_{10} = - b_{10}\wedge b_{10}\wedge b_{11} =0,
	\]
	i.e., $\Re b_1=0$.

	Thus, in the case $s=\tr\Psi\neq 0$, a solution $(a,\Psi)$ of~\eqref{Eq_SW2gen} yields a trivialisation of $\cL$; Moreover,  $a$ is the product connection \wrt  this trivialisation, and $b_1$ is a purely imaginary harmonic 1-form.   
\end{proof}

Let  $(a,\Psi)$ be a solution of~\eqref{Eq_SW2gen} with $E=\slS$ and $b$ being the Levi--Civita connection.
Regarding $(a,\Psi)$ as a pair $(a, b)$ just like in \autoref{Lem_SW2PSL2RBijection} and recalling \autoref{Lem_FlatPSL2RRedSW}, we obtain  that  $A = a+(\pi^*b)i$ is a flat stable $\PSL(2,\R)$-connection. 
Assign to $(a,\Psi)$ the holonomy representation
\[
\rho_{a,\Psi}\colon \pi_1(M)\to\PSL(2,\R)
\]
of $A$. Of course, $\rho_{a,\Psi}$ is well-defined up to the conjugation in $\PSL(2,\R)$ only.

\section{Involutions and a homeomorphism between moduli spaces}
\label{Sect_Involutions}

Pick $\a\in H^2(M;\Z)$ and let  
\begin{equation*}
	\cM_{SW2}(\a):=\big \{ (a,\Psi) \mid (a,\Psi) \text{ solves } \eqref{Eq_SW2gen}\text{ and } c_1(\cL) = \a\,\big\}/C^\infty(M;\U(1))
\end{equation*}
be the moduli space of solutions of \eqref{Eq_SW2gen} with a fixed line bundle $\cL$. 
Denote also 
\begin{equation*}
	\cM_{SW2}:=\bigsqcup_{\a\,\in\, H^2(M;\,\Z)} \cM_{SW2}(\a).
\end{equation*}
We shall see below that the above union is in fact finite. 

It is well-known that the moduli space of solutions of the classical Seiberg--Witten equations  (just one spinor) is equipped with an involution~\cite{Morgan96_SWequations}*{Sect.\,6.7}. 
This construction carries over essentially verbatim to the present setting and yields an involution $\sigma\colon \cM_{SW2}\to \cM_{SW2}$ such that $\sigma\colon \cM_{SW2}(\a)\to \cM_{SW2}(-\a)$ is a homeomorphism for each $\a\in H^2(M;\Z)$.  

However, it is easier and more convenient to describe this involution in terms of solutions of~\eqref{Eq_SFSl2Creduced}. 
Thus, set 
\begin{equation*}
	\sM(\a):= \big \{ (a,b) \text{ solves } \eqref{Eq_SFSl2Creduced} \text{ and } c_1(\cL)=\a\,\big \}/C^\infty(M;\U(1))\qquad\text{and}\qquad \sM:=\bigsqcup_{\a}\sM(\a). 	
\end{equation*}
Letting $\cL^\vee$ denote the dual of $\cL$, define $\sigma\colon \cA(\cL)\times \Om^1(M;\cL) \to \cA(\cL^\vee)\times\Om^1(M;\cL^\vee)$ by 
\begin{equation*}
	\sigma(a, b)=\big (  a^\vee, \,  b^\vee\,\big ),	
\end{equation*}
where $a^\vee$ is the connection dual to $a$ and if $b$ equals  $\om\otimes s$ locally, then $ b^\vee = \bar \om \otimes \langle \cdot, s\rangle$.
Then $\sigma$ preserves the space of solutions of~\eqref{Eq_SFSl2Creduced} and yields a homeomorphism $\sigma\colon \sM(\a)\to \sM(-a)$. 
This defines implicitly an involution, still denoted by the same letter, on $\cM_{SW2}$.

\begin{lem}
	\label{Lem_InvolutionEqn}
	Let $Q^c$ be a principal $\PSL(2,\C)$-bundle equipped with a principal $\PSU(2)$-subbundle $Q$. 
	Let  $g_0\in \PSL(2,\C)$ denote a non-trivial representative in $N(\PSL(2,\R))/\PSL(2,\R)\cong \Z/2\Z$, where $N(\PSL(2,\R))$ is the normalizer of $\PSL(2,\R)$ in $\PSL(2,\C)$. 
	Let $A$ be a flat stable $\PSL(2,\C)$ connection which reduces to a principal $\PSL(2,\R)$-bundle $Q^r\subset Q^c$. 
	If $P=Q^r\cap Q$ and  $A|_{P} = a + (\pi^*b)i$, then $A$ also reduces to $Q^r\cdot g_0$ and the corresponding $\U(1)$-bundle $P^\vee:=(Q^r\cdot g_0)\cap Q$ is naturally isomorphic to the $\U(1)$-bundle of $\cL^\vee$. 
	Moreover, the restriction of $A|_{P^\vee}$  corresponds to the pair $\sigma(a, b)$.    
\end{lem}  
\begin{proof}
	Choose 
	\begin{equation*}
		g_0= 
		\begin{pmatrix}
			0 & i\\
			i & 0
		\end{pmatrix}.
	\end{equation*}
	We have $P^\vee=(Q^r\cdot g_0)\cap Q = P\cdot g_0$, since $g_0\in \PSU(2)$. 
	Moreover, the natural map $P\to P^\vee$, \ $p\mapsto p\cdot g_0$ is bijective and fiber-preserving, however maps $p\cdot z$ to $p\cdot zg_0=(p\cdot g_0)\cdot\bar z$, where $z\in \U(1)=\PSL(2,\R)\cap \PSU(2)$.
	Therefore, $P^\vee$ is isomorphic to the principal  $\U(1)$-bundle corresponding to $\cL^\vee$. 
	Moreover, if $A|_{P}=a + (\pi^*b) i$, then $A|_{P\cdot g_0} =g_0^{-1} \big (a+ (\pi^*b) i \big ) g_0 = -a + (\overline {\pi^* b}) i$, where $\overline {b_2\xi_2 + b_3\xi_3} = b_2\xi_2 - b_3\xi_3$.
	This finishes the proof of this lemma.  
\end{proof}

Furthermore, let 
\begin{equation*}
	\cR(M):=\Hom\big ( \pi_1(M),\; \PSL(2,\R)  \big)/\text{conj.}
\end{equation*} 
be the $\PSL(2,\R)$-representation variety of $M$, where the quotient is taken by the conjugation in $\PSL(2,\R)$.  
Observe that $\cR(M)$ is also equipped with an involution, which I also denote by $\sigma$. 
Indeed, keeping to the notations of  \autoref{Lem_InvolutionEqn}, $\sigma$ is defined by  $\sigma(\rho) = g_0\rho g_0^{-1}$. 
Moreover, the natural map
\begin{equation*}
	\hat H\colon \cM_{SW2}\to \cR(M), \qquad \hat H\big ([a,\Psi] \big ) = [\rho_{a, \Psi}]  
\end{equation*}
is equivariant, that is $\hat H\comp\sigma = \sigma\comp \hat H$. 
Thus, $\hat H$ yields a map
\begin{equation*}
	H\colon \cM_{SW2}/\langle \sigma\rangle \longrightarrow \cR(M)/ \langle \sigma\rangle.
\end{equation*}

A representation $\pi_1(M)\to \PSL(2,\R)$ is said to be irreducible, if no point in $\CP^1$ is fixed by all elements of $\pi_1(M)$.
Here $\PSL(2,\R)$ acts on $\CP^1$ via the standard embedding $\PSL(2,\R)\subset \PSL(2,\C)$, where the latter group acts via M\"obius transformations.
Denote by $\cR^{irr}(M)$ the subspace of classes of irreducible representations and set 
\begin{equation*}
	\cM^{irr}_{SW2}:=\big \{  (a,\Psi)\mid (a,\Psi)\text{ solves }\eqref{Eq_SW2gen}\text{ and } \rho_{a,\Psi}\text{ is irreducible} \big \}/C^\infty(M;\U(1)).
\end{equation*}
Notice that if $\Psi\equiv 0$,  the corresponding $\PSL(2,\R)$-representation cannot be irreducible so that $\cM_{SW2}^{irr}$ consists of irreducible solutions in the sense of the Seiberg--Witten theory.  

Since $\PSL(2,\C)$ acts freely on the space of irreducible $\PSL(2,\C)$-representations, $\sigma$ acts freely on $\cR^{irr}(M)$ and, therefore, on $\cM_{SW2}^{irr}$ too. 
Clearly, by restriction we obtain a map
\begin{equation}
	\label{Eq_Hirr}
	H\colon \cM^{irr}_{SW2}/\langle \sigma\rangle \longrightarrow \cR^{irr}(M)/ \langle \sigma\rangle
\end{equation}
still denoted by the same letter. 

\begin{proposition}
	\label{Prop_HisHomeo}
	$H$ is a homeomorphism. 
\end{proposition}
\begin{proof}
	By a result of Donaldson\footnote{In~\cite{Donaldson:87_TwistedHarmonicMaps} the base manifold is of dimension two, however this is not really used in the proof. }~\cite{Donaldson:87_TwistedHarmonicMaps}, an irreducible  representation $\rho\colon\pi_1(M)\to\PSL(2,\R)\subset \PSL(2,\C)$ yields a flat stable $\PSL(2,\C)$-connection $A_\rho$ on a $\PSL(2,\C)$-bundle $Q^c$. 	
	Pick a reduction of the structure group of this bundle to $\PSU(2)$. 
	
	Since the holonomy of $A_\rho$ is contained in $\PSL(2,\R)$, there is a $\PSL(2,\R)$-subbundle $Q^r\subset Q^c$ such that $A_\rho$ reduces to $Q^r$, cf. Remark~\ref{Rem_UniquenessPSL2Rsubbundle}. 
	Let $q_0\in Q^r$ be a base point such that the holonomy representation of $A_\rho$ equals $\rho$. 
	If the holonomy representation of $A_\rho$ with respect to any other base point $q_0\cdot g$, $g\in \PSL(2,\C)$, is also contained in $\PSL(2,\R)$, then $g\in N(\PSL(2,\R))$.
	Hence, $A_\rho$ reduces to exactly two $\PSL(2,\R)$-subbundles, namely $Q^r$ and $Q^r\cdot g_0$. 
	If $A_\rho |_{Q^r}=a_\rho +(\pi^*b_\rho)i$, then $[a_\rho,b_\rho]\in  \sM/\langle \sigma\rangle$ is well defined.  
	By Lemma~\ref{Lem_SW2PSL2RBijection} we obtain a unique class $[a,\Psi]\in \cM_{SW2}/\langle \sigma\rangle$ such that $H\big ([a,\Psi]\big ) =[\rho]$.
	This finishes the proof of this proposition. 	
\end{proof}

It is a well-known fact from the Seiberg--Witten theory, that the moduli space of the classical Seiberg--Witten monopoles is non-empty for finitely many spin$^c$-structures only, see for example~\cite{Morgan96_SWequations}*{Thm.\,5.2.4}.
To the best of my knowledge, it is not known whether this holds for the Seiberg--Witten equations with two or more spinors.
However, \autoref{Prop_HisHomeo} can be used to show that for $E=\slS$ this finiteness property still holds.

\begin{corollary}
	\label{Cor_FinManyBasicClasses}
	If $E=\slS$ is equipped with the Levi--Civita connection, then the set 
	\begin{equation*}
		\bigl \{ \a\in H^2(M;\Z)\mid  \cM_{SW2}(\a)\neq\varnothing  \bigr \}	
	\end{equation*}
	is finite. 
\end{corollary}
\begin{proof}
	Choose a basis $(\sigma_1,\dots, \sigma_{b_2})$ of $H_2(M)/\mathrm{Tor}$ and represent each $\sigma_j$ by an embedded surface $\Sigma_j$.
	Pick $\a \in H^2(M;\Z)$ and a line bundle $\cL$ such that $c_1(\cL) = \a$.
	If   $\cM_{SW2}(\a)\neq\varnothing$, by the Milnor--Wood inequality~\cite{Milnor58_ExistenceOfFlatConn} we obtain 
	\[
	\bigl |c_1(\cL|_{\Sigma_j})\bigr|=
	\bigl| \langle \sigma_j, c_1(\cL) \rangle \bigr|\le \mathrm{genus}(\Sigma_j)-1.
	\]	
	This implies the statement of this corollary.
\end{proof}

Let me finish this section with some examples of  $\PSL(2;\R)$ representation varieties---hence, also of the moduli space of solutions to~\eqref{Eq_SW2gen}---paying a particular attention to the case of integral homology spheres, since this will be of interest in the next section.

	For the Brieskorn homology sphere 
\begin{equation}
\label{Eq_BrieskornHomSpheres}
\Sigma(p,q,r):=\bigl\{  z\in\C^3\mid\; z_1^p + z_2^q + z_3^r =0\;\bigr \}\cap S^5,
\end{equation}
where $p,q,r$ are coprime positive integers, 
the $\PSL(2,\R)$--representation variety is finite~\cite{KitanoYamaguchi16_SL2RBrieskorn_Arx}. 
Moreover, all non-trivial representations are irreducible.  
In particular, for $M=\Sigma(p,q,r)$ the space $\cM_{SW2}$ is compact. 

The Brieskorn homology spheres can be also used to construct examples of homology spheres, for which the $\PSL(2,\R)$ representation varieties are non-compact as follows. 
Assume $M_1$ and $M_2$ are two arbitrary closed three-manifolds each admitting an irreducible representation $\rho_i\colon \pi_1(M_i)\to \PSL(2,\R)$. 
Assume also for the sake of simplicity that each $\rho_i$ is rigid, i.e., that $[\rho_i]$ is an isolated point in $\cR(M_i)$.
By the van Kampen theorem, the fundamental group of the connected sum $M:=M_1\# M_2$ is the free product  $\pi_1(M_1)\ast \pi_1(M_2)$ so that we obtain a non-trivial family of representations $\rho_A\colon \pi_1(M_1)\ast \pi_1(M_2)\to \PSL(2,\R)$ corresponding to $(\rho_1, A\rho_2 A^{-1})$, where $A\in\PSL(2,\R)$ is a parameter.   
Notice that $\rho_A$ is conjugate to $\rho_B$ if and only if $A=B$. 
It is easy to see that $\rho_{A_k}$ converges in $\cR(M)$ if and only if $A_k$ converges in $\PSL(2,\R)$. 
Hence, $\cR(M)$ is non-compact.
In particular, this yields the following: If $M$ is the connected sum of two Brieskorn homology spheres --- and, hence, also a homology sphere --- then  the moduli space of solutions of~\eqref{Eq_SW2gen} on $M$  is non-compact.
Hence~\cite{HaydysWalpuski15_CompThm_GAFA},~\eqref{Eq_SW2genLim} admits non-trivial solutions on such $M$ for any background metric, even though $M$ does not support any non-trivial honest harmonic 1-form.

 \section{Blow up sets}
 \label{Sect_BlowUpSets}

 Let $Z$ be a blow up set and $(a,\Psi)$ a solution of~\eqref{Eq_SW2genLim} just as in Definition~\ref{Defn_BlowUpSet}. 
In view of \autoref{Lem_SW2PSL2RBijection}, the case $\tr\Psi\neq 0$ is easy to analyse, hence from now on I will assume that $\tr \Psi =0$.
Just like in Section~\ref{Sect_BasicConstr}, $\Cl^{-1}(\Psi)=:b\in\Om^1(M\setminus Z; \cL)$ satisfies
\begin{equation}
	\label{Eq_SW2Lim}
	(d_a + d_a^*)\,b=0,\quad \sigma(b\wedge b) =0 \qquad\text{over }M\setminus Z. 
\end{equation}
Of course, $|b|$ also extends to $M$ as a continuous function and $|b|^{-1}(0) = Z$. 

If we consider $b$ as a section of $\Hom (TM;\, \cL)$, then the condition $\sigma(b\wedge b)$ implies that  the image of $b$ in $\cL$ is one-dimensional over $M\setminus Z$. 
Hence, we can define the real line bundle $\cI:=\Im b\subset \cL$ over $M\setminus Z$ and consider $b$ as a 1-form with values in $\cI$.
Hence,  a solution of \eqref{Eq_SW2Lim} can be thought of a $\Z/2$ harmonic 1-form in the following sense.

\begin{defn}[\cite{Taubes13_PSL2Ccmpt}]
	Let $(\om, Z,\cI)$ be a triple such that 
	\begin{itemize}[itemsep=-4pt, topsep=2pt]
		\item $Z$ is a closed subset of $M$ of Hausdorff dimension at most one;
		\item $\cI$ is a real Euclidean line bundle over $M\setminus Z$ equipped with a canonical Euclidean connection $\nabla$; 
		\item $\om\in \Gamma\big ( M\setminus Z;\, T^*M\otimes\cI\big )$ satisfies $d\om =0 = d^*\om $;
		\item $\int_{M\setminus Z} |\nabla\om |^2<\infty$;
		\item $|\om|$ extends as a H\"older-continuous function to all of $M$ and $|\om|^{-1}(0) = Z$. 
	\end{itemize}
	Under these circumstances $(\om, Z,\cI)$ is called a $\Z/2$ harmonic 1-form. 
\end{defn}

%


From now on I assume that $M$ is an integral homology sphere and that the blow up set $Z$ is a 1-dimensional submanifold of $M$, that is a link in $M$. 
Let  $Z_1$ denote the union of all components of $Z$ such that $\mathcal I$ is non-trivial on the meridian of any component.
Notice that $Z_1\neq\varnothing$, since otherwise $\cI$ would be trivial so that $\om$ would extend as a non-trivial harmonic 1-form to all of $M$.

The double covering branched along a link plays an important r\^ole in what follows below, therefore let me pause for a while to recall this construction.
Let me pick $Z_1$ as a branching set and for this reason it is convenient to choose an orientation of each component of $Z_1$, albeit this will play a minor r\^ole below. 
Denoting by $N_{Z_1}$ a (small) tubular neighborhood of  $Z_1$ so that the closure $\bar N_{Z_1}$ consists of  $k$ solid tori, where $k$ is the number of connected components of $Z_1$.  
The Mayer--Vietoris sequence applied to $M\setminus Z_1$ and $N_{Z_1}$ yields 
\begin{equation*}
	H_1(M\setminus {Z_1})=\Z\mu_1\oplus\dots\Z\mu_k,
\end{equation*}
where each $\mu_j$ is a meridian of a connected component of $Z_1$, see for example~\cite{Prasolov07_HomologyTh}*{PP.\,367--368}.

Denote $Y:=M\setminus N_{Z_1}$ and consider the homomorphism 
\begin{equation}
\label{Eq_DoubleCovHom}
\pi_1(Y)\to H_1(Y)\cong H_1(M\setminus Z_1)\to \Z \to \Z/2,
\end{equation}
where the first arrow stands for the abelianization homomorphism, the second one is given by $\sum a_j\mu_j\mapsto \sum a_j$, and the last one is the canonical projection. 
The kernel of this homomorphism is an index $2$ subgroup of $\pi_1(M\setminus N_{Z_1})$. 
Let $\hat Y$ denote the corresponding (unbranched) double covering of $Y$.
This is an oriented manifold with boundary consisting of $k$ disjoint tori. 
We glue in $\bar N_{Z_1}$ via a homomorphism $h\colon \del \hat Y\to \del \bar N_{Z_1}$ such that $h(\hat \mu_j) = \mu_j$, where $\hat\mu_j$ is the preimage in $\hat Y$ of $\mu_j$. 
The result of this is a smoth oriented manifold $\hat M$, which is called the double covering of $M$ branched along $Z_1$.

Notice that the construction yields naturally a smooth map $\pi\colon \hat Y\to Y$, which can be extended as a map $\pi\colon \hat M\to M$ such that $\pi$ is 2-to-1 over $M\setminus Z_1$ and 1-to-1 over the branching set $Z_1$. 
One way to fix such extension is as follows.  
Chose an identification $ \bar N_{Z_1}\cong \sqcup \big ( S^1\times D\big )$, where $D=\{ |z|\le 1 \}$ is the unit disc in $\C$.  
This identification can be chosen so that the canonical involution on $\hat Y$ extends as a map $(\theta, z)\mapsto (\theta, -z)$ on each connected component of $\bar N_{Z_1}$ and we can set $\pi(\theta, z)=(\theta, z^2)$.
Albeit this particular extension is very common in the literature, notice that this is by no means unique.  
I shall come back to this point in the proof of \autoref{Prop_DoubleCover}   below.

\medskip

 Recall  that a representation $\rho\colon G\to \PSL(2;\R)$ of a group $G$ is called metabelian, if $\rho\bigl ([G,G]\bigr )$ is an abelian subgroup. 
 
 \begin{proposition}
 	\label{Prop_DoubleCover}
 	Let $M$ be an integral homology sphere. 
 	For a solution $(a,b)$ of \eqref{Eq_SW2Lim} denote $A:=a + bi$, which is a flat $\PSL(2;\R)$-connection on $M\setminus Z$.
 	Then the following holds:
 	\begin{enumerate}[(i),itemsep=-4pt, topsep=2pt]
 		\item\label{It_PullbackItrivial} $\pi^*\mathcal I$ is trivial over $\hat M\setminus \hat Z_1$; 
 		\item\label{It_HolonomyAbelianNontriv}	 The holonomy representation of $\pi^*A$ is abelian and non-trivial.
 		\item\label{It_HolonomyMetabelian} 	
 		The holonomy representation of $A$ is metabelian.
 		\item\label{It_FirstBettiNumber}
 		The first Betti number of $\hat M$ is positive.
 	\end{enumerate}
 \end{proposition}
 \begin{proof}
 	The flat bundle $\cI$ corresponds to a homomorphism $\pi_1(M\setminus Z_1)\to \Z/2$, which factors through $H_1(M\setminus Z_1)$, since $\Z/2$ is abelian:
 	\begin{equation*}
 		\tau\colon\pi_1(M\setminus Z_1)\to H_1(M\setminus Z_1)\to \Z/2.
 	\end{equation*}
 	By inspection, this coincides with~\eqref{Eq_DoubleCovHom} taking into account $\pi_1(Y)\cong \pi_1(M\setminus Z_1)$. 
 	Then the pull-back $\pi^*\cI$ is a flat bundle corresponding to the homomorphism
 	\begin{equation*}
 		\pi_1(M_2\setminus \hat Z_1) \xrightarrow{\ \pi_*\ } \pi_1(M\setminus Z_1)\xrightarrow{\ \tau\ } \Z/2,
 	\end{equation*}
 	which is trivial, since the image of $\pi_*$ is in the kernel of $\tau$. 
 	This proves~\textit{\ref{It_PullbackItrivial}}.
 	
 	To prove~\textit{\ref{It_HolonomyAbelianNontriv}},  consider the following  basis of $\fpsl(2,\R)$:
 	\[
 	\xi_1=
 	\begin{pmatrix}
 	0 & -1\\ 1 & \phantom{-}0
 	\end{pmatrix},
 	\quad
 	\xi_2=
 	\begin{pmatrix}
 	0 & 1\\ 1 & 0
 	\end{pmatrix},
 	\qand\quad
 	\xi_3=
 	\begin{pmatrix}
 	1 & \phantom{-}0\\ 0 & -1
 	\end{pmatrix},
 	\]
 	cf.~\eqref{Eq_BasisSU2}. 
 	Notice that the 1-parameter group generated by $\xi_1$ is isomorphic to $\U(1)$, whereas the one parameter group corresponding to any non-trivial $\xi\in \mathrm{span} \{ \xi_2, \xi_3 \}$ is isomorphic to $\R$.
 	
 	By~\textit{\ref{It_PullbackItrivial}}, $\pi^*a$ is a flat connection with trivial holonomy on a trivial line bundle. 
 	Therefore, after applying a gauge transformation we can assume that $a$ is the product connection on the product line bundle. 
 	
 	Furthermore, since $\cI$ is a subbundle of $\cL$, $\pi^*\cI = \underline\R$ is a subbundle of $\pi^*\cL$, i.e., we have a trivialization of $\pi^*\cL$ over $\hat M\setminus\hat Z_1$. 
 	Hence, $\pi^*b$ can be viewed as an $\R^2=\mathrm{span}\{ \xi_2, \xi_3 \}$-valued 1-form and $\pi^*A = d + (\pi^*b)i $ can be viewed as a connection on the product $\PSL(2;\R)$-bundle.
 	Then $F_{\pi^*A} =0$ together with $(d+d^*)\, \pi^*b =0$ imply that 
 	\begin{equation}
 	\label{Eq_AuxOmXi}
 	\pi^*b = \om\,\xi
 	\end{equation} 
 	for some fixed $\xi\in \mathrm{span}\{ \xi_1, \xi_2 \}$, where $\om$ is a 1-form. 
 	Moreover,  $\om$ is closed, and therefore the holonomy of $\pi^*A$ along a loop $\gamma$ is given by 
 	\[
 	\Hol_\gamma (\pi^*A) = \exp \Bigl ( \int_\gamma \om\; \xi\Bigr ).
 	\]
 	In other words, the holonomy of $\pi^*A$ is determined by the periods of $\om$. 
 	Even though $\om$ is only continuous along $\hat Z_1$, I claim that the de Rham cohomology class of $\om$ is well-defined and non-trivial.
 	The proof of this claim, however,  requires some background material, which is introduced first.
 	I follow the line of argument of~\cite{Wang93_ModuliSpacesInvolutions}*{Sect.\,1.2} in this part.
 	
 	Thus, pick a component of $\hat Z_1$ and identify its neighbourhood with $S^1\times D$ so that the canonical involution acts by multiplication by $-1$ on $D$  just as above.
 	It is convenient to extend $\pi\colon \hat Y\to Y$ to $\hat M$ so that on each component of $\bar N_{Z_1}$ it has the following form 
 	\begin{equation*}
 		\hat \pi\colon S^1\times \C\to S^1\times \C,\qquad (\theta, z)\mapsto (\theta, z^2/|z|).
 	\end{equation*} 
 	Notice that $\hat\pi$ is smooth away from $\hat Z_1$ but is only Lipschitz near $\hat Z_1$. 
 	Nevertheless, $\hat \pi$ has the following important advantage over $\pi$:  For any smooth metric $g_{0}$ on $\hat M$ there exist positive constants $c_1$ and $c_2$ such that the inequalities 
 	\begin{equation*}
 		c_1\, g_{0}\le \hat\pi^*g \le c_2\, g_{0}
 	\end{equation*}
 	hold on $\hat M\setminus \hat Z_1$ in the sense of quadratic forms. 
 	Hence,   $\hat \pi^*g$  is a Lipschitz metric on $\hat M$ in the sense of~\cite{Teleman83_IndexSignOnLipschitzMflds}*{Sec.\,3}. 
 	In particular, the spaces of $L^2$-functions (forms) on $\hat M$ with respect to $\hat\pi^*g$ and $g_0$ coincide and the corresponding norms are equivalent. 
 	
 	Furthermore, denote by $\cH^2_1$ the space of all 1-forms $\om$ on $\hat M$ such that the following holds: $\om$ is smooth on $\hat M\setminus \hat Z_1$, $\om\in L^2(T^*\hat M)$, and $d\om =0 =d \bigl (*_{\hat\pi^*g}\,\om\bigr )$ pointwise on $\hat M\setminus \hat Z_1$, cf.~\cite{Wang93_ModuliSpacesInvolutions}*{Def.\,6}. 
 	We have a version of Hodge theorem in this setting~\cite{Teleman83_IndexSignOnLipschitzMflds}*{Thm.\,4.1}, i.e., a natural isomorphism $\cH^2_1\to H^1(\hat M;\R)$.

 	With this understood, $\hat\pi^*\om$ is clearly a non-trivial harmonic 1-form with respect to $\hat\pi^*g$ taking values in $\hat\pi^*\cI\cong \underline\R$.
 	In particular, by the version of the Hodge theorem mentioned above, $[\hat\pi^*\om]=[\pi^*\om]$ represents a non-trivial class in $H^1(\hat M;\R)$.  
 	This finishes the proof of~\textit{\ref{It_HolonomyAbelianNontriv}} and proves~\textit{\ref{It_FirstBettiNumber}} as well. 	
 	
 	To prove~\textit{\ref{It_HolonomyMetabelian}}, notice that we have the following short exact sequence
 	\[
 	1\longrightarrow \pi_1(\hat M)\longrightarrow \pi_1(M\setminus Z_1)\xrightarrow{\ \;\tau\ \;} \{ \pm 1 \}\longrightarrow 1 
 	\]
 	where $\tau$ sends meridians of each components of $Z_1$ to $-1$. 
 	Combining this with~\textit{\ref{It_HolonomyAbelianNontriv}}, we obtain that  the holonomy of $A$ lies in the subgroup $H$ generated by matrices of the form $\exp (t\xi)$, where $\xi\in \mathrm{span}\{ \xi_1, \xi_2 \}$ is as in~\eqref{Eq_AuxOmXi}, and a matrix $B\in \{ \exp(t\xi_1) \}$ such that $B^2 = \mathbbm 1$. 
 	Concretely,
 	\begin{equation*}
 		B= 
 		\begin{pmatrix}
 			0 & -1\\
 			1 & \phantom{-}0
 		\end{pmatrix}.
 	\end{equation*}
 	Here I think of $\PSL(2;\R)$ as $\SL(2;\R)/\pm 1$ so that $B^2$ is the identity element indeed.

 	It is easy to check directly that $H$ is a metabelian subgroup of $\PSL(2,\R)$, i.e., $[H, H] = \{ \exp(t\xi) \}$ is abelian. 
 	Thus, the holonomy representation of $A$ is metabelian as claimed. 
 \end{proof}

  \begin{proof}[\textbf{\textit{Proof of \autoref{Thm_AlexPolyBlowUpSet}}}]
  	\label{Pf_AlexPoly}
 	It is well-known that if the Alexander polynomial of a link in an integral homology sphere does not vanish at $-1$, then the corresponding double branched covering is a rational homology sphere~\cite{Lickorish97_IntroKnotThry}*{Cor.\,9.2}  (see also~\cite{Kawauchi96_SurveyKnotTheory_Book}*{Sect.\,5.5} and~\cite{PoudelSaveliev17_LinkHomolEquivGaugeTh}*{Prop.\,3.1}).
 	For a one-component link, i.e., a knot, the double branched covering is always a rational homology sphere~\cite{Lickorish97_IntroKnotThry}*{P.\,95}.
 \end{proof}

 Let me note in passing that examples of $\Z/2$ harmonic 1-forms on homology 3-spheres such that the corresponding branching set $Z$ is a link will appear elsewhere.

 \appendix
 \section{An algebraic property of $\mu$}
\label{Sect_AlgPtyMu}

The only purpose of this appendix to provide some details of the calculation  concerning the quadratic map $\mu$ used in the proof of \autoref{Lem_SW2PSL2RBijection}. 
More precisely, I claim that the diagram 
\[
\begin{CD}
	T_\C^*M @>\Cl>> \End_0(\slS)\\
	@V\chi(\cdot\wedge\cdot)VV @VV\mu V\\
	\Lambda^2T^*M\otimes \R i @>\Cl>> \End_0(\slS)
\end{CD}
\]
commutes, where the horizontal arrows represent the Clifford multiplication, the left vertical arrow represents a combination of the standard symplectic product on $\C\cong \R^2$ and the wedge-product, and the right vertical arrow is given by~\eqref{Eq_Mu}.

It suffices to show that the following diagram 
\begin{equation}
	\label{Eq_CommDiagrMu}
\begin{CD}
(\R^3)^*\otimes\C @>\Cl>> \End_0(\C^2)\\
@V\chi(\cdot\wedge\cdot)i\,VV @VV\mu V\\
\Lambda^2(\R^3)^*\otimes \R i @>\Cl>> \End_0(\C^2)
\end{CD}
\end{equation}
commutes. 
To this end, let $(e_1^*, e_2^*, e_3^*)$ be the standard basis of $(\R^3)^*$. 
Since the Clifford multiplication maps $e_j^*$ to $\eta_j$, which is defined by~\eqref{Eq_BasisSU2},  we have
\begin{equation*}
	b:=\sum_{j=1}^3 z_je_j^*\ \longmapsto
	\begin{pmatrix}
		z_3i \quad& -z_1-iz_2\\[5pt]
		z_1 - z_2i\quad & -z_3i
	\end{pmatrix}
	=:\Psi.
\end{equation*}
Then a straightforward computation yields 
\begin{equation*}
	\mu(\Psi) = 
	\begin{pmatrix}
	-\big (  z_1\bar z_2 - z_2\bar z_1\big )i \quad&  ( z_3\bar z_1 - z_1\bar z_3) i -z_3\bar z_2 + z_2\bar z_3\\[5pt]
	( z_3\bar z_1 - z_1\bar z_3)i + z_3\bar z_2 - z_2\bar z_3 \quad&  \big (  z_1\bar z_2 - z_2\bar z_1\big )i
	\end{pmatrix}.
\end{equation*}

Furthermore, writing $z_j= x_j + y_ji$, we obtain
\begin{equation*}
	\chi(b\wedge b) = \big ( x_1y_2 - x_2y_1 \big )\, e_1^*\wedge e_2^* + \big ( x_1y_3 - x_3y_1 \big )\, e_1^*\wedge e_3^* + \big ( x_2y_3 - x_3y_2 \big )\, e_2^*\wedge e_3^*.
\end{equation*}
Hence, 
\begin{equation*}
	\Cl\big ( \chi(b\wedge b)i  \big)=
	2\begin{pmatrix}
		-(x_1y_2-x_2y_1) \quad & x_1y_3-x_3y_1- (x_2y_3-x_3y_2)i\ \\[5pt]
		x_1y_3-x_3y_1+ (x_2y_3-x_3y_2)i \quad& x_1y_2-x_2y_1 \
	\end{pmatrix}.
\end{equation*}
Therefore, we obtain $\mu(\Psi) = \Cl\big (\chi(b\wedge b)i\,\big )$   by inspection.
This establishes the commutativity of~\eqref{Eq_CommDiagrMu}.

\bibliography{references}

\end{document}